\documentclass[a4paper,10pt]{article}

\usepackage[T1]{fontenc}											
\usepackage[utf8]{inputenc}										
\usepackage[italian,english]{babel}							
\usepackage[english]{varioref}									
\usepackage{epigraph}												
\usepackage{xparse}													

\usepackage{amsmath}												
\usepackage{amssymb}											
\usepackage{amsthm}												
\usepackage{commath}												
\usepackage{mathtools}											
\usepackage{bigints}													
\usepackage{braket}													
\usepackage{fge}														

\newtheorem{Theorem}{Theorem}
\newtheorem{Proposition}[Theorem]{Proposition}
\newtheorem{Lemma}[Theorem]{Lemma}

\newtheorem{Definition}{Definition}

\RequirePackage{dsfont}

\renewcommand{\subset}{\subseteq}			
\newcommand{\1}{\mathds 1}					

\newcommand{\Z}{\mathbb{Z}}					
\newcommand{\R}{\mathbb{R}}					
\newcommand{\Cx}{\mathbb{C}}				

\DeclareMathOperator{\SL}{SL}

\DeclareMathOperator{\SU}{SU}

\DeclareMathOperator{\End}{End}

\DeclareMathOperator{\Tan}{T}					

\newcommand{\rec}{\frac{1}}						
\newcommand{\inv}[2][1]{#2^{-#1}}					

\DeclareMathOperator{\ad}{ad}		
\newcommand{\contr}{\! \cdot \!}			

\NewDocumentCommand{\dert}{o m}{\dif \IfNoValueF{#1}{^{#1}}_{#2}}																						
\NewDocumentCommand{\derp}{o m}{\partial\IfNoValueF{#1}{^{#1}}_{#2}}																				
\NewDocumentCommand{\dbydt}{o m G{}}{\frac{\dif \IfNoValueF{#1}{^{#1}} {#3}}{\dif {{#2}}\IfNoValueF{#1}{^{#1}}}}						
\NewDocumentCommand{\dbydp}{o m G{}}{\frac{\partial\IfNoValueF{#1}{^{#1}} {#3}}{\partial { {#2} }\IfNoValueF{#1}{^{#1}}}}	

\newcommand{\Hilb}{\mathcal H}								
\newcommand{\smth}[1][\infty]{\mathcal C^#1}			
\DeclareMathOperator{\Lint}{L}									

\newcommand{\ModSp}{\mathcal M}							
\RequirePackage{mathrsfs}												
\newcommand{\Lqnt}{\mathscr L}										
\newcommand{\Teich}{\mathcal T}												
\DeclareMathOperator{\Mod}{Mod}											

\newcommand{\Nabla}{\boldsymbol \nabla}								
\newcommand{\HWC}{\tilde \Nabla{}}										
\newcommand{\EHWC}{\HWC^{\End}}										

\let\bar\overline
\newcommand{\lalg}[1]{\mathfrak{#1}}									
\newcommand{\su}{\lalg{su}}

\RequirePackage{mathrsfs}												

\newcommand{\subsupk}[2]{#1_{#2}^{(k)}}
\newcommand{\Triv}{P}																
\newcommand{\triv}[1][l]{\Triv^{(#1)}}											
\newcommand{\Sexp}[1][ l ]{{\mathcal{R}}^{(#1)}}						
\newcommand{\Pexp}[1][ l ]{A^{(#1)}}										
\newcommand{\Obs}{\operatorname{Obs}^{(l)} (\mathcal{R})}	
\newcommand{\Coeff}[2]{\alpha_{#1}^{(#2)}}								
\newcommand{\Xoeff}[2]{X_{#1}^{(#2)}}										
\newcommand{\Troeff}[2]{\tilde{X}_{#1}^{(#2)}}							
\newcommand{\Trunc}[1][l]{\operatorname{Trunc}^{(l)}}				

\newcommand{\dT}{{\dif\,}_{\Teich}}												
\newcommand{\dF}{{\dif \,}^{F}_{\Teich}}										
\newcommand{\nablatr}{\nabla^{\textup{Tr}}}								
\newcommand{\FHWC}{\tilde{\mathcal{D}}}														

\newcommand{\Toep}[1]{\subsupk{T}{#1}}									
\newcommand{\curveop}[1][f]{\mathcal{C}_{#1}}						
\newcommand{\BTstar}{\star^{\text{BT}}}									

\newcommand{\DiffD}{\operatorname{D}_{k}}									
\newcommand{\DiffA}{\mathscr{A}_{k}}									

\usepackage[bookmarks,colorlinks=false]{hyperref}

\title{Asymptotic properties of the Hitchin-Witten connection}
\author{Jørgen Ellegaard Andersen and Alessandro Malus\`{a}
\thanks{Supported in part by the center of excellence grant "Centre for quantum geometry of Moduli Spaces" DNRF95, from the Danish National Research Foundation.}}
\date{}

\begin{document} 

\maketitle

\phantomsection
\pdfbookmark[1]{Abstract}{Abstract}
\begin{abstract}
	We explore extensions to $\SL(n,\Cx)$-Chern-Simons theory of some results obtained for $\SU(n)$-Chern-Simons theory via the asymptotic properties of the Hitchin connection and its relation to Toeplitz operators developed previously by the first named author.
	We define a formal Hitchin-Witten connection for the imaginary part $s$ of the quantum parameter $t = k+is$ and investigate the existence of a formal trivialisation.
	After reducing the problem to a recursive system of differential equations, we identify a cohomological obstruction to the existence of a solution.
	We explicitly find one for the first step, in the specific case of an operator of order $0$, and show in general the vanishing of a weakened version of the obstruction.
	We also find a solution of the whole recursion in the case of a surface of genus $1$.
\end{abstract}

\phantomsection
\pdfbookmark[1]{Contents}{Contents}
\tableofcontents

\section{Introduction}

In the context of the $\SU(n)$-Chern-Simons theory, a deformation quantisation can be obtained from the study of the asymptotic properties of Toeplitz operators and the Hitchin connection as was done in~\cite{And05,And06,And12}.
Let $\Sigma$ denote a smooth, oriented, closed surface of genus $g \geq 2$.
The construction starts from the setting of~\cite{Hit90,ADPW91,And12}, considering the moduli space $\ModSp^{n,d}$ of flat $\SU(n)$-connections on $\Sigma$ with fixed holonomy $e^{2 \pi i d/n} \1$ around a puncture.
If $\Sigma$ comes with a Riemann surface structure $\sigma$, then one can run Kähler quantisation on the moduli space using the induced complex structure and the Chern-Simons pre-quantum line bundle $\Lqnt$.
For $k$ a positive integer, the level-$k$ quantum Hilbert space is $H^{0}_{\sigma} (\ModSp^{n,d} , \Lqnt^{k})$, i.e. that of holomorphic sections of $\Lqnt^{k} := \Lqnt^{\otimes k}$, with the $\Lint^{2}$ product.
To every smooth function $f$ on $\ModSp^{n,d}$ is associated a Toeplitz operator $\Toep{\sigma,f}$, which consists of the multiplication by $f$ followed by the orthogonal projection to $H^{0}_{\sigma} (\ModSp^{n,d} , \Lqnt^{k})$.
When $\ModSp^{n,d}$ is smooth, it follows from the works of Bordeman, Karabegov, Meinrenken, and Schlichenmaier~\cite{BMS94,Sch98,Sch00,KS01} that, for functions $f$ and $g$, there exists an all-order asymptotic expansion
\begin{equation*}
	\Toep{\sigma,f} \circ \Toep{\sigma,g} \sim \sum_{l = 0}^{\infty} \Toep{\sigma,c^{(l)} (f,g)} k^{-l} \, .
\end{equation*}
Moreover, the $c^{(l)}$'s act on $f$ and $g$ as bi-differential operators and define a deformation product $\BTstar_{\sigma}$ on $\smth(\ModSp^{n,d}) [[\hbar]]$, called the Berezin-Toeplitz product.

These constructions depend on the Riemann surface structure $\sigma$ on $\Sigma$ via its isotopy class: as this parameter varies, the objects obtained above are parametrised by the Teichmüller space $\Teich$.
For each level $k$, the quantum Hilbert spaces form the so-called Verlinde bundle $H^{(k)}$, which hosts the projectively flat Hitchin connection $\Nabla$~\cite{Hit90,ADPW91,And12, AG14}.
As $\sigma$ varies over $\Teich$, one may view the collection $\Toep{f}$ of the Toeplitz operators associated to a function $f$ as a section of $\End(H^{(k)})$.
It is then of interest to determine to what extent this depends on $\sigma$, which is measured by its covariant derivative with respect to the endomorphism connection $\Nabla^{\End}$ induced by $\Nabla$.
Although said derivative is not zero in general, it is proven by Andersen in~\cite{And12} that, for every vector $V$ tangent to the Teichmüller space, there exists a unique asymptotic expansion of $\Nabla^{\End}_{V} \Toep{f}$ as
\begin{equation*}
	\Nabla^{\End}_{V} \Toep{f} \sim \sum_{l=1}^{\infty} \Toep{\tilde{D}_{V} (f)} ( 2k + \lambda )^{-l} \, ,
\end{equation*}
where $\lambda = 2 \operatorname{GCD}(n,d)$.
The coefficients $\tilde{D}_{V}(f)$ are linear in $V$ and act on $f$ as differential operators, and they define a connection $D$ on the trivial bundle $\Teich \times \smth(\ModSp^{n,d})[[ \hbar ]]\to\Teich$, called the formal Hitchin connection.
Furthermore, Andersen proves that, having trivial holonomy, the formal connection admits a trivialisation.
Combined with the compatibility of $D$ with the products $\BTstar_{\sigma}$, this produces an identification of the deformation quantisations arising from different values of $\sigma$.

It should be mentioned at this point that the construction above is compatible with the mapping class group of the surface in the following sense.
The simultaneous action of $\Mod_{\Sigma}$ on $\Teich$ and $\ModSp^{n,d}$ induces one on the Verlinde bundle, and it follows from the equivariance of the Kähler structure that the action preserves the Hitchin connection.
A projective representation of $\Mod_{\Sigma}$ can then be obtained at every level $k$ on the space of covariantly constant sections of the Verlinde bundle, which was proven to be asymptotically faithful in \cite{And06}.
It is argued in~\cite{And12} that the trivialisation of $D$ can also be chosen to be $\Mod_{\Sigma}$-invariant, and so is the resulting $\sigma$-independent deformation quantisation.
		
Another interesting aspect of the asymptotic properties of Toeplitz operators regards the relation between the approaches to Chern-Simons theory via geometric quantisation and the Witten-Reshetikhin-Turaev TQFT.
The data of a simple closed curve $\gamma \subset \Sigma$ and a representation $\rho$ of $\SU(n)$ defines a quantum operator in each of these two viewpoints.
On the one hand, the holonomy function $h_{\gamma,\rho}$ determines a Toeplitz operator, call it $\Toep{\gamma,\rho}$.
On the other hand, $\gamma$ determines a framed knot in $\Sigma \times [0,1]$; once decorated with $\rho$, this defines an operator in the Reshetikhin-Turaev TQFT.
These two operators can be compared using the chain of isomorphisms of~\cite{AU1,AU2,AU3,AU4,Las98}, and it was proven by Andersen in~\cite{And10} that their difference vanishes asymptotically at the first order for $k \to \infty$.

The goal of the present work is to extend part of the picture recalled above to the case of $\SL(n,\Cx)$, following Witten's approach~\cite{Wit91,AG14}.
Consider the moduli space $\ModSp^{n,d}_{\Cx}$ of reductive flat $\SL(n,\Cx)$-connections over $\Sigma$, again with prescribed holonomy around a puncture.
Its smooth locus carries a complex symplectic form, and the real moduli space $\ModSp^{n,d}$ sits inside $\ModSp^{n,d}_{\Cx}$ as a symplectic subspace.
For every complex level $t = k + is$, with $k$ a positive integer, classical Chern-Simons theory defines on $\ModSp^{n,d}_{\Cx}$ a pre-quantum line bundle $\Lqnt^{(t)}$, which restricts to $\Lqnt^{k}$ on $\ModSp^{n,d}$.
Witten considers on the complex moduli space a real polarisation, whose leaves intersect $\ModSp^{n,d}$ transversely.
Since the polarised sections are completely determined by their restriction to $\ModSp^{n,d}$, Witten proposes that one identifies the level-$t$ quantum Hilbert space with $\Lint^{2} (\ModSp^{n,d} , \Lqnt^{k})$.
However, the polarisation depends again essentially on the Riemann surface structure on $\Sigma$, so the resulting Hilbert spaces should be collected into an infinite-rank bundle $\Lint^{2} (\ModSp^{n,d},\Lqnt^{k})\times\Teich\to\Teich$.
In analogy with the situation for $\SU(n)$, the dependence on the Teichmüller parameter is measured by a projectively flat connection $\HWC$, referred to as the Hitchin-Witten connection.
We address the problem of studying the asymptotic properties of $\HWC$ in the imaginary part $s$ of the level, aiming to define an analogue of the formal Hitchin connection and a trivialisation.

Although we keep the Chern-Simons theory as our main motivation, the problem can actually be set in broader and more abstract terms, as in~\cite{AG14}.
In the first part of the paper, we shall denote by $(\ModSp,\omega)$ a symplectic manifold, together with an integrable almost complex structure $J$ which depends smoothly on a parameter $\sigma$ in a complex manifold $\Teich$.
In the presence of a pre-quantum line bundle $\Lqnt$ on $\ModSp$, one may define the trivial bundle
\begin{equation*}
	\Hilb^{(t)} := \Lint^{2} (\ModSp , \Lqnt^{k}) \times \Teich \to \Teich \, .
\end{equation*}
Under the hypotheses of~\cite{AG14}, summarised in Section~\ref{sec:AG14}, a projectively flat connection $\HWC$ is defined, taking the form
\begin{equation*}
	\HWC = \nablatr + \rec{2t} b - \rec{2 \overline{t}} \overline{b} + \dT F
\end{equation*}
for appropriate forms $b$, $\overline{b}$ and $\dT F$ on $\Teich$, valued in finite-order differential operators acting on sections of $\Lqnt^{k}$.
We call this the Hitchin-Witten connection, as it agrees with the one obtained by Witten~\cite{Wit91} for the complex Chern-Simons theory.
Its projective flatness follows from Hitchin's proof of the same property for his connection in the case of $\SU(n)$~\cite{Hit90,AG14}.
If, moreover, a group $\Gamma$ acts on both $\ModSp$ and $\Teich$ in such a way that $J$ is equivariant, the Hitchin-Witten connection is also $\Gamma$-invariant.

Unlike in the case of Kähler quantisation, each of the fibres $\Hilb^{(t)}_{\sigma}$ is closed under multiplication by a smooth function $f$ on $\ModSp$.
Therefore, there is in this case no close analogue of the Toeplitz operators.
One may define a curve operator $\curveop[f]$ as the multiplication by $f$ as specified by the formula
\begin{equation*}
	\curveop (\psi) = f \psi.
\end{equation*}
The terminology is motivated by the ideas of~\cite{And10}, where Toeplitz operators are related to the curve operators from the Witten-Reshetikhin-Turaev TQFT.
One might then attempt to study the asymptotic properties of these curve operators and their Hitchin-Witten covariant derivatives, and look for asymptotic expansions with coefficients of the same kind.
However, this would be too restrictive, for a number of reasons.
First of all, recall that, in our motivational example of $\SL(n,\Cx)$-Chern-Simons theory, sections on $\ModSp$ should be regarded as polarised objects on a larger space $\ModSp_{\Cx}$.
Functions on $\ModSp$ also correspond to the polarised ones on the complex space, and the action of $\curveop[f]$ matches that of the pre-quantum operator associated to the polarised extension of $f$ to $\ModSp_{\Cx}$.
While justifying the commutativity of curve operators, this shows that the relevant algebra to quantise is that of functions on $\ModSp_{\Cx}$ rather than $\ModSp$.
It would then seem natural that one considers asymptotic expansions for finite-order differential operators, and allow these as coefficients even for the expansion of curve operators.
Moreover, it will be clear from the later discussion that an expansion with curve operators as coefficients need not exist in general.

For a fixed value of $k$, we consider the algebra $\DiffA = \DiffD[[\inv{s}]]$ of formal power series in $\inv{s}$ with finite-order differential operators acting on sections of $\Lqnt^{k}$ as coefficients.
Our first result is the existence of a unique formal Hitchin-Witten connection.

\begin{Theorem}
	\label{thm:formal_connection}
	Under the conditions listed in Section~\ref{sec:AG14}, the trivial bundle $\DiffA \times \Teich \to \Teich$ admits a unique formal connection
	\begin{equation*}
		\FHWC = \nablatr + \sum_{l = 0}^{\infty} \FHWC^{(l)}
	\end{equation*}
	characterised by the property that, for every vector field $V$ on $\Teich$, differential operator $D \in \DiffD$ and positive integer $L$, one has
	\begin{equation}
		\label{eq:FHWC_asymptotic}
		\EHWC_{V} D - V[D] - \sum_{l = 0}^{L} \FHWC_{V}^{(l)} D = o \bigl( s^{-L} \bigr)
		\quad \text{ for $s \to \infty$,}
	\end{equation}
	where $\EHWC$ is the connection induced by $\HWC$ on the endomorphism bundle, and the convergence holds in a sense specified in Section~\ref{sec:convergence}.
	Said connection is flat and can be written explicitly as
	\begin{equation*}
		\FHWC (D) = \nabla^{F} (D) - \rec{2k} \sum_{l = 1}^{\infty} (ik)^{l} \bigl[ b - (-1)^{l} \overline{b} , D \bigr] s^{-l} \, .
	\end{equation*}
	Under our assumptions on $\Gamma$, moreover, $\FHWC$ is also invariant under its action.
\end{Theorem}

Next we address the problem of finding a trivialisation for this formal connection, i.e. a map $\DiffD \to \DiffA$ sending every differential operator $D$ to a series
\begin{equation*}
	\mathcal{R} (D) = \sum_{l = 0}^{\infty} \Sexp (D) s^{-l}
\end{equation*}
such that $\Sexp[0] (D) = D$ and $\FHWC \mathcal{R} (D) = 0$.
Written explicitly, this condition boils down to the recursive relations
\begin{equation}
	\label{eq:introrec}
	\dF \Sexp[l] (D) = \rec{2k} \sum_{n = 1}^{l} \left( ik \right)^n \left[ b - (-1)^{n} \overline{b} , \Sexp[l - n] (D) \right] := \Obs \, ,
\end{equation}
for every non-negative integer $l$, where $\dF$ is a suitably defined twisted exterior differential.
It is apparent from the equation that an obstruction to the existence of a solution comes in general from the differential of the right-hand side.
However, we prove the following statement.

\begin{Theorem}
	\label{thm:no_Sobstruction}
	Suppose $D = \Sexp[0] (D) , \dots , \Sexp[l-1] (D)$ are given, which satisfy the first $l$ steps of the recursion.
	Then the right-hand side in~\eqref{eq:introrec} is closed:
	\begin{equation*}
		\dF \Obs = \dF \Biggl( \sum_{n=1}^{l} (ik)^{n} \bigl[ b - (-1)^{n} \overline{b} , \Sexp[l-n] (D) \bigr] \Biggr) = 0 \, .
	\end{equation*}
\end{Theorem}

The statement shows that $\Obs$ defines a class in $H^{1} (\Teich , \DiffD)$, and a solution of the differential equation exists if and only if said class vanishes.
While this is the case in Chern-Simons theory, the Teichmüller space being contractible, this does not automatically conclude our discussion, as it is of crucial importance that the solution be $\Gamma$-invariant.
If $\Sexp[0] (D) , \dots , \Sexp[l-1] (D)$ satisfy this condition, then a class is defined in $H^{1}_{\Gamma} ( \Teich , \DiffD)$, the first cohomology group of the complex of $\Gamma$-invariant $\DiffD$-valued forms on $\Teich$.
It is a consequence of the definitions that a $\Gamma$-invariant solution of the $l$-th step exists if and only if this class vanishes; in particular a $\Gamma$-invariant trivialisation of $\FHWC$ exists if $H^{1}_{\Gamma} ( \Teich,\DiffD) = 0$.

\begin{Theorem}
	\label{thm:curveop}
	In the case when $D = \curveop$ is the operator of multiplication by a smooth function $f$, independent on the Teichmüller parameter, we find a $\Gamma$-invariant first-order solution $\Sexp[1] (\curveop)$, for which moreover
	\begin{equation*}
		\EHWC \bigl( \curveop + \Sexp[1] (\curveop) \inv{s} \bigr) = o (\inv{\abs{t}}) \, .
	\end{equation*}
\end{Theorem}

In order to illustrate our motivation for restricting to the case when $k$ is fixed, we briefly discuss an analogous recursion for the formal parameters $1/t$ and $1/\overline{t}$.
We conclude that the cohomological obstruction arising in that situation does not vanish in general, not even for a zero-order operator $\curveop$.

We then focus again on Chern-Simons theory in the specific situation of a surface $\Sigma$ of genus $1$.
Although this case was excluded from the general discussion presented above, the key constructions can be carried out using adapted arguments.
In fact, the Hitchin and Hitchin-Witten connections are still defined in this case, and they are furthermore flat, the latter having an explicit trivialisation proposed by Witten.
We find a sequence of $\Mod_{\Sigma}$-invariant operators $\Pexp (D)$ satisfying a recursive relation similar to the desired one for the $\Sexp (D)$s, which motivates us to look for a solution of the original recursion of the form
\begin{equation*}
	\Sexp (D) = \sum_{r = 0}^{l} \Coeff{r}{l} \Pexp[l-r] (D) \, .
\end{equation*}
By looking for solutions of this kind specifically, the problem reduces to a numeric recursion on the coefficients $\Coeff{r}{l}$, which leads to our final result.

\begin{Theorem}
	\label{thm:introtrivialisation}
	Every sequence of complex numbers $\Coeff{r}{r}$ uniquely extends to a solution of the aforementioned numeric recursion.
	As a consequence, there exist infinitely many solution of the form above, resulting in $\Mod_{\Sigma}$-invariant trivialisations of the formal Hitchin-Witten connection.
	Moreover, any two such objects are related by the multiplication by a power series in $\inv{s}$ with constant coefficients, and one particular solution of this kind arises as a formal expansion of the explicit trivialisation of $\HWC$.
\end{Theorem}

\subsection*{Plan of the exposition}

We open this paper by briefly summarising the main facts from~\cite{AG14} that we are going to use.
In particular, we list the hypotheses used in the cited work for defining the Hitchin-Witten connection and proving its projective flatness.
This is the content of Section~\ref{sec:AG14}.

In Section~\ref{sec:convergence} we discuss the matters of convergence for the expansions that we are going to consider, and specify the meaning of the asymptotic limits.

The main results of this paper are detailed in Section~\ref{sec:FHWC}, in which we discuss the formal Hitchin-Witten connection for the imaginary part of the quantum parameters in the general situation.
After proving Theorem~\ref{thm:formal_connection}, we argue that a trivialisation of the Hitchin-Witten connection is equivalent to a perturbative series of differential operators which are covariantly constant asymptotically to every order.
We then proceed to set up the recursive equations determining the trivialisation and prove Theorems~\ref{thm:no_Sobstruction} and~\ref{thm:curveop}.
We conclude the section with a proof that the same recursive approach fails when working with formal Laurent series in $t$ and $\overline{t}$, thus motivating our choice of working with $s$ instead.

Section~\ref{sec:genus1} is dedicated to the specific case of the moduli space of flat $\SU(n)$-connections on a surface of genus $1$.
After discussing the adaptations needed for defining the Hitchin-Witten connection, we show its explicit trivialisation following Witten.
We then proceed to considering the recursion as in Section~\ref{sec:FHWC} and exhibiting a family of solutions, finally proving Theorem~\ref{thm:introtrivialisation}.

In the Appendix we briefly recall the notion of the total symbol of differential operators, and specify the meaning of smoothness for operator-valued differential forms.

\section[Families of complex structures; the Hitchin-Witten connection]{Families of complex structures and the Hitchin-Witten connection}
\label{sec:AG14}

In this section we set the notations and summarise some of the definitions and results of~\cite{AG14}.

Unless otherwise specified, $(\ModSp,\omega)$ will denote a symplectic manifold with $H^{1}(\ModSp,\R) = 0$, a pre-quantum line bundle $\Lqnt$ and a family of complex structures parametrised by a complex manifold $\Teich$.
By this we mean a map $J \colon \Teich \to \smth(\ModSp , \End(\Tan \ModSp))$, smooth in the sense of Appendix~\ref{sec:operator_forms}, valued in integrable almost complex structures on $\ModSp$.
The tensor $J_{\sigma}$ is also assumed to form a Kähler structure together with $\omega$ and $g_{\sigma} = \omega \contr J_{\sigma}$, and to have no non-trivial holomorphic functions or vector fields.
In the following we shall often neglect the subscript $\sigma$ in $J_{\sigma}$ and $g_{\sigma}$.
Also, we shall denote by $\dT$ the de Rham differential on $\Teich$, to avoid confusion with that on $\ModSp$, and similarly write $\partial_{\Teich}$ and $\overline{\partial}_{\Teich}$ for its  $(1,0)$ and $(0,1)$ parts.

With $J$ we assume given a Ricci potential, i.e. a real-valued function $F$ on $\ModSp \times \Teich$ such that, for fixed $\sigma \in \Teich$, on has
\begin{equation*}
	2 i \partial \overline{\partial} F = \rho - \lambda \omega \, ,
\end{equation*}
where $\lambda$ is a constant parameter and $\rho$ denotes the Ricci form of the Kähler metric.
This requires that $\lambda \omega/2 \pi$ represents the first Chern class of $(\ModSp,\omega)$, and it is conversely implied by this condition via the $i \partial \overline{\partial}$-lemma if $\ModSp$ is compact.
Such a potential is unique under our assumptions up to addition of a constant; a choice may be fixed in the case when $\ModSp$ is compact by fixing its mean value to $0$.
In the following we shall use $\dF = \dT + \dT F$ to denote the exterior differential twisted with $\dT F$.
Notice that, while for non-compact $\ModSp$ there is no preferred choice of $F$ in general, the ambiguity induced on $\dF$ is by a term valued in central differential operators, and hence vanishes when acting on operator-valued forms.

If $V$ is a vector on $\Teich$, it is easily seen by differentiating $J^{2} = - \1$ that $V[J]$ and $J$ anti-commute, so $V[J]$ maps $\Tan^{(1,0)} \ModSp$ to $\Tan^{(0,1)} \ModSp$ and conversely.
This implies that the variation of the inverse $\tilde{g}$ of $g$ decomposes into parts $(2,0)$ and $(0,2)$, since the inverse $\tilde{\omega}$ of $\omega$ is of type $(1,1)$, and so
\begin{equation*}
	- V[\tilde{g}] =: \tilde{G}(V) = \tilde{\omega} \contr V[J] = G(V) + \overline{G}(V) \, .
\end{equation*}
On the other hand, $\tilde{G}$ can be thought of as a form on $\Teich$ valued in tensor fields on $\ModSp$ and decomposed into its  $(1,0)$ and $(0,1)$ parts accordingly.
We shall assume that $J$ enjoys the following two properties.

\begin{Definition}
	We say that $J$ is holomorphic if $G$ and $\overline{G}$ are of type $(1,0)$ and $(0,1)$ as a differential form on $\Teich$, respectively.
	We say that moreover it is rigid if, for every $V$ on $\Teich$, the tensor field $G(V)$ on $\ModSp$ is holomorphic.
\end{Definition}

With the metric $g$ comes a Laplace-Beltrami operator $\Delta$ acting on sections of $\Lqnt^{k}$, which depends on $\sigma$ through both $\tilde{g}$ and the Levi-Civita connection.
It can be showed that the variation of this operator is given by
\begin{equation*}
	- V[\Delta] = \nabla^{2}_{\tilde{G}(V)} + \nabla_{\delta \tilde{G} (V)} = : \Delta_{\tilde{G}(V)} \, .
\end{equation*}
The holomorphic and anti-holomorphic derivatives are given by the corresponding expressions with $G$ and $\overline{G}$ in place of $\tilde{G}$.

After these premises we are ready to define the Hitchin-Witten connection.

\begin{Definition}
	Consider the forms on $\Teich$, valued in differential operators acting on sections of (tensor powers of) $\Lqnt$, given by
	\begin{equation*}
		\begin{gathered}
			b(V) = \Delta_{G(V)} + 2 \nabla_{G(V) \cdot \dif F} - 2 \lambda V'[ F ] , \\
			\overline{b}(V) = \Delta_{\overline{G}(V)} + 2 \nabla_{\overline{G}(V) \cdot \dif F} - 2 \lambda V'' [ F ] \, .
		\end{gathered}
	\end{equation*}
	For a fixed parameter $t = k + is$ with $k \in \Z_{>0}$ and $s \in \R$, define the level-$t$ Hitchin-Witten connection on the trivial bundle $\smth(\ModSp , \Lqnt^{k}) \times \Teich \to \Teich$ as
	\begin{equation}
		\label{eq:HWC}
		\HWC = \nablatr + \rec{2t} b - \rec{2 \overline{t}} \overline{b} + \dT F \, ,
	\end{equation}
	where $\nablatr$ is the trivial connection.
	We denote by $\EHWC$ the connection induced on the endomorphism bundle from $\HWC$.
\end{Definition}

\begin{Theorem}[\cite{Wit89,AG14}]
	Under the assumptions listed above, $\HWC$ is projectively flat, i.e. the curvature $F_{\HWC}$ takes values in central differential operators.
\end{Theorem}

Notice that, as a consequence of this, $\EHWC$ is in fact flat.

If $\Sigma$ is a closed, oriented, smooth surface of genus $g \geq 2$, (the smooth part of) the moduli space $\ModSp^{n,d}$ carries a symplectic structure given by the Narasimhan-Atiyah-Bott-Goldman form.
The Teichmüller space $\Teich$ parametrises a smooth family of complex structures on $\ModSp^{n,d}$.
Under the identification of $\Tan_{[A]} \ModSp^{n,d}$ with the twisted cohomology group $H^{1}_{A} (\Sigma,\lalg{su}(n))$, the complex structure corresponding to $\sigma \in \Teich$ corresponds to the Hodge $*$-operator.
Excluding the special case of $g = n = 2$ and $d$ even, the known properties of $\ModSp^{n,d}$ ensure that the hypotheses listed above are verified, thus leading to the Hitchin-Witten connection of~\cite{Wit91}.

\section[Meaning of the convergence]{Meaning of the convergence of the asymptotic expansions}
	\label{sec:convergence}

	The asymptotic results in~\cite{Sch00,And12} are phrased in terms of the $\Lint^2$ operator norm, which makes good sense in the context of the finite-dimensional spaces of holomorphic sections.
	In the situation at hand, however, we need to consider the space of all smooth sections, on which differential operators are typically unbounded with respect to the $\Lint^2$ norm.
	
	One natural way around this would be to use Sobolev norms, but some care is needed, as their definition for sections of smooth bundles relies in general on various choices, in an essential way.
	As a matter of fact, the resulting norms are equivalent but different, so what is intrinsically well defined is the Sobolev topology alone.
	In order to make sense of the comparison of norms of operators acting on different spaces some normalisation might be needed; we shall not address this problem here.
	However, if the asymptotic limit is taken for $s \to \infty$ while keeping $k$ fixed, the space on which the operators are acting is also fixed, so one can pick a choice and consistently carry out the limiting process with respect to it.
	
	Alternatively, one can use strong convergence and say that an operator $D$ decays at a given rate if the $\Lint^{2}$ norm of $D \psi$ does, for every smooth section $\psi$ with $\norm{\psi} = 1$.
	Of course this approach still requires that the bundle, hence $k$, is fixed, and it does not define a norm for the operators.
	On the other hand, it has the advantage of carrying an intrinsic meaning in terms of the Hilbert space structure relevant for geometric quantisation.
		
	Another way to phrase the matters of convergence is via the symbols of differential operators, see Appendix~\ref{sec:symbols}.
	If $D$ is a finite-order differential operator acting on sections of $\Lqnt^{k}$, its total symbol $\sigma(D)$ is a formal (i.e. not necessarily homogeneous) tensor field on $\ModSp$, whose $\Lint^{\infty}$ norm is defined via the Riemannian metric.
	Since the correspondence between differential operators and totally symmetric tensor fields is a bijection, this gives a norm $\norm{\cdot}_{T}$ on the space of operators, which makes sense independently on $t$.
	
	Suppose now that $k$ is fixed, and that an operator $D$ depends on $s$ as a Laurent polynomial, i.e. $D$ has the form
	\begin{equation*}
		D = \sum_{n = N_{0}}^{N_{1}} D^{(n)} s^{-n} \, ,
	\end{equation*}
	where each $D^{(n)}$ is independent of $s$.
	One can then attempt to argue that each of the norms considered above is bounded by $C \abs{s}^{-N_{0}}$ for some positive real constant $C$.
	Whenever this is the case, we shall write
	\begin{equation*}
		D = o \bigl( \abs{s}^{-\alpha} \bigr)  \quad \text{for $s \to \infty$,}
	\end{equation*}
	for every $\alpha < N_{0}$, without any further reference to the norm.
	In a similar fashion, if $D$ depends on $t$ and its symbol can be expressed as a Laurent polynomial in $t$ and $\overline{t}$, then $\norm{D}_{T}$ is also bounded by a power of $\abs{t}$; we shall write
	\begin{equation*}
		D = o \bigl( \abs{t}^{-\alpha} \bigr) \quad \text{for $t \to \infty$.}
	\end{equation*}

	In view of the formal approach used in the next section, and in order to make the above statements even more precise, it is convenient to introduce the following graded algebras.
	
	\begin{Definition}
		\label{def:DiffA}
		For each fixed $k \in \Z_{>0}$ let $\DiffD$ be the algebra of finite-order differential operators acting on $\smth(\ModSp,\Lqnt^k)$, and denote
		\begin{equation*}
			\DiffA := \DiffD[[\inv{s}]]
		\end{equation*}
		regarded as a graded algebra.
	\end{Definition}
	
	As endomorphisms of $\DiffA$ we shall only consider those arrsing as formal power series with coefficients in $\End(\DiffD)$.
	In the setting of the formal Hitchin connection, all the transformations of Toeplitz operators are obtained by acting on the function defining them as differential operators.
	Similarly, we shall require that the coefficients of the endomorphisms of $\DiffA$ should act as differential operators on their symbols.
	
%
	
\section{The formal approach}
	\label{sec:FHWC}

	\subsection{The formal Hitchin-Witten connection}

	If $D$ depends smoothly on $\sigma$, its Hitchin-Witten covariant derivative reads
	\begin{equation*}
		\EHWC D = \dF D + \rec{2t} [ b , D ] - \rec{2 \overline{t}} [\overline{b} , D ] \, .
	\end{equation*}
	We now study the existence of a formal connection on $\DiffA \times \Teich$ reproducing this covariant derivative asymptotically.
	\begin{Definition}
		By a formal connection on $\DiffA \times \Teich$ we mean a sum
		\begin{equation*}
			\FHWC = \nabla^{F} + \sum_{l=1}^{\infty} \FHWC^{(l)} s^{-l} \, ,
		\end{equation*}
		where each $\FHWC^{(l)}$ is a $1$-form on $\Teich$ with values in $\End( \DiffD)$.
	\end{Definition}
	
	\begin{proof}[Proof of Theorem~\ref{thm:formal_connection}]
		First of all, consider the Taylor expansions of $\inv{t}$ and $\inv{\overline{t}}$ in $s$ at $s = \infty$ and obtain
		\begin{equation}
			\label{eq:taylor}
			\rec{t} = \rec{k + is} = - \rec{k} \sum_{n=1}^{\infty} \left( \frac{ik}{s} \right)^n \, ,
				\qquad
			\rec{\bar{t}} = \rec{k - is} = - \rec{k} \sum_{n=1}^{\infty} \left( - \frac{ik}{s} \right)^n \, .
		\end{equation}
		These converge for $\abs{s} > k$; in particular the error of each $L$-th truncated sum decays faster than $\abs{s}^{-L}$ for $s \to \infty$, since $k$ is fixed.
		Based on this, we choose
		\begin{equation*}
			\FHWC^{(l)} (D) : = - \frac{(ik)^{l}}{2k} \left[ b - (-1)^l \overline{b} , D \right] \, .
		\end{equation*}
		As a consequence of the convergence of~\eqref{eq:taylor}, for any positive integer $L$ one can notice that for any norm $\norm{\cdot}$ one has
		\begin{equation*}
			\begin{aligned}
				& \biggl\| \EHWC_{V} (D) - \nabla^{F}_{V} (D) - \sum_{l = 1}^{L} \FHWC^{(l)}_{V} s^{-l} (D) \biggr\| \leq \\[.4em]
				\leq{} & \Biggl\| \! \Biggl( \rec{2t} \!+\! \rec{2k} \sum_{l=1}^{L} {\biggl( \frac{ik}{s} \biggr)\!\!}^l \Biggr) \! \bigl[ b (V) , D \bigr] \Biggr\| \!
				+ \Biggl\| \! \Biggl( \rec{2\overline{t}} \!+\! \rec{2k} \sum_{l=1}^{L} {\biggl( - \frac{ik}{s} \biggr)\!\!}^l \Biggr) \! \bigl[ \overline{b} (V) , D \bigr] \Biggr\| \! = \\[.4em]
				={}& o \bigl( \abs{s}^{ - L} \bigr) \Bigl\|\bigl[ b , D \bigr] \Bigr\| + o \bigl( \abs{s}^{ - L} \bigr) \Bigl\|\bigl[ \overline{b} , D \bigr] \Bigr\| = o \bigl( \abs{s}^{ - L} \bigr) \, .
			\end{aligned}
		\end{equation*}
		This goes verbatim for the case when all the operators are applied to a smooth section $\psi$, proving the existence.
		The uniqueness is implied by the asymptotic condition~\eqref{eq:FHWC_asymptotic}.
		Indeed, if two such formal connections $\FHWC$ and $\FHWC'$ are given, their $0$-order terms agree by assumption.
		Suppose, on the other hand, that $\FHWC^{(L)} \ne \FHWC'^{(L)}$ for some $L$, which we assume to be minimum.
		Then for every $D \in \DiffD$ one has
		\begin{equation*}
			\begin{split}
				\norm{ \bigl( \FHWC^{(L)} - \FHWC'^{(L)} \bigr) (D) } s^{-L} 
				\leq &\ \biggr\| \biggl( \EHWC_{V} - \nabla^F_{V} - \sum_{l = 1}^{L} \FHWC^{(l)}_{V} s^{-l} \biggr) (D) \biggr\| + \\
				+ &\ \biggl\| \biggl( \EHWC_{V} - \nabla^F_{V} - \sum_{l = 1}^{L} \FHWC'^{(l)}_{V} s^{-l} \biggr) (D) \biggr\| \, .
			\end{split}
		\end{equation*}
		The last expression decays faster than $s^{-L}$, but since $\bigl\lVert \bigl( \FHWC^{(L)} - \FHWC'^{(L)} \bigr) (D) \bigr\rVert$ does not depend on $s$, it has to be zero for every $D$, which contradicts $\FHWC^{(L)} \ne \FHWC'^{(L)}$.
		
		Flatness can be proven in a similar fashion.
		Indeed, the curvature of $\FHWC$ is expressed by
		\begin{equation*}
			\sum_{l = 1}^{\infty} \biggl( \dF \FHWC^{(l)} + \rec{2} \sum_{n + m = l} \left[ \FHWC^{(n)} \wedge \FHWC^{(m)} \right] \biggr) s^{-l} \, .
		\end{equation*}
		More explicitly, the $l$-th coefficient of its action on operators is given by the commutator with
		\begin{equation*}
			- \frac{(ik)^{l}}{2k} \dF \big( b - (-1)^{l} \overline{b} \big) + \frac{(ik)^{l}}{8k^2} \sum_{n + m = l} \left[ \left( b - (-1)^{n} \overline{b} \right) \wedge \left( b - (-1)^{m} \overline{b} \right) \right] .
		\end{equation*}
		Since $[b \wedge b]$ and $[ \overline{b} \wedge \overline{b} ]$ take values in central differential operators (see Proposition 4.6 in~\cite{AG14}), the whole curvature is
		\begin{equation}
			\label{eq:formal_curvature}
			\sum_{l = 1}^{\infty}  \biggl( - \frac{(ik)^{l}}{2k} \dF \big( b - (-1)^{l} \overline{b} \big) + \frac{(ik)^{l}}{4k^2} \!\! \sum_{n + m = l} \!\! (-1)^n \left[ b \wedge \overline{b} \right] \biggr) s^{-l} \, .
		\end{equation}
		For comparison, the curvature of $\EHWC$ is given by the commutator with
		\begin{equation}
			\label{eq:exact_curvature}
			\dF \left( \rec{2t} b - \rec{2 \overline{t}} \overline{b} \right) - \rec{4 \abs{t}^2} [ b \wedge \overline{b} ] \, .
		\end{equation}
		It can be seen that the coefficients in~\eqref{eq:formal_curvature} give the Laurent expansions of those in~\eqref{eq:exact_curvature} at $s = \infty$.
		Similar arguments to those used to prove the existence of $\FHWC$, combined with the flatness of $\EHWC$, imply that for every positive integer $L$ and every operator $D$ one has
		\begin{equation*}
			\Biggl[ \sum_{l = 1}^{L} \biggl( \frac{(ik)^{l}}{2k} \dF \big( b - (-1)^{l} \overline{b} \big) - \frac{(ik)^{l}}{4k^2} \sum_{n + m = l} (-1)^n \bigl[ b \wedge \overline{b} \bigr] \biggr) s^{-l} , D \Biggr] \!=\! o(s^{-L}) .
		\end{equation*}
		On the other hand, this expression depends on $s$ as a Laurent polynomial of order at most $L$, hence it vanishes.
		Therefore, all the truncations of the curvature of $\FHWC$ are zero, which proves the flatness.
	\end{proof}
	
	Of course the result on the curvature may also be proven by direct application of the algebraic relations found in~\cite{AG14} in the process of proving the projective flatness of the Hitchin and Hitchin-Witten connection.
	
	We emphasise that the symbols of $\FHWC^{(l)}_{V} (D)$ depend as differential operators on those of $D$, so this connection acts on the fibres by endomorphisms of $\DiffA$ of the kind described after Definition~\ref{def:DiffA}.
	
	\begin{Definition}
		We shall refer to the connection of Theorem~\ref{thm:formal_connection} as the formal Hitchin-Witten connection, and indicate it with the notation $\FHWC$.
	\end{Definition}

\subsection{The recursion by differential equations}

	We now consider an operator $D$ depending smoothly on $\sigma$, and look for an asymptotically $\EHWC$-parallel expansion
	\begin{equation}
		\label{eq:formal_series}
		\mathcal{R} (D) := \sum_{l = 0}^{\infty} \Sexp (D) s^{-l} \, ,
	\end{equation}
	where $\Sexp[0] (D)$ should be $D$.
	We shall pose the asymptotic requirement that, for every vector $V$ on $\Teich$ and every positive integer $L$, the covariant derivative of the $L$-th truncation of the series should decay faster than $s^{-L}$, or explicitly
	\begin{equation}
		\label{eq:convergence}
		\EHWC \biggl( \sum_{l = 0}^{L} \Sexp (D) s^{-l} \biggr) = o(s^{-L}) \, .
	\end{equation}
	Notice that such an expansion cannot exist unless $\dF D = 0$.
	Indeed, this is the only contribution to the covariant derivative of degree $0$ in $s$, so it cannot be counter-balanced by any other terms.
	We shall then assume from now on that this condition is indeed satisfied.
	Notice that, in the case of a curve operator $\curveop$, this is equivalent to $f$ being independent of $\sigma$.
	
	We set the problem in terms of the formal Hitchin-Witten connection by using the following fact.
	
	\begin{Proposition}
		\label{lemma:asymptotic}
		The asymptotic condition~\eqref{eq:convergence} is satisfied if and only if, as a formal power series, \eqref{eq:formal_series} is covariantly constant with respect to $\FHWC$.
	\end{Proposition}
	
	\begin{proof}
		Written explicitly, the formal covariant derivative of a power series as in~\eqref{eq:formal_series} reads
		\begin{equation}
			\label{eq:formal_derivative}
			\FHWC \bigl( \mathcal{R} (D) \bigr) = \sum_{l = 0}^{\infty} \Bigl( \dF \Sexp (D) + \sum_{n = 1}^{l} \FHWC^{(n)} \bigl( \Sexp[l-n] (D) \bigr) \Bigr) s^{-l} \, .
		\end{equation}
		
		Let now $V$ and $L$ be fixed.
		By the defining property of the Hitchin-Witten connection, one has
		\begin{equation*}
			\begin{split}
				\EHWC_{V} \biggl( \sum_{l = 0}^{L} \Sexp (D) s^{-l} \biggr) ={}& \nabla^{F}_{V} \biggl( \sum_{l = 0}^{L} \Sexp (D) s^{-l} \biggr) + \\
				& + \sum_{n = 1}^{L} \FHWC^{(n)}_{V} \biggl( \sum_{l = 0}^{L} \Sexp (D) s^{-l} \biggr) s^{-n} + o(s^{-L}) \, .
			\end{split}
		\end{equation*}
		All the terms of degree higher than $L$ in $\inv{s}$ on the right-hand side may be absorbed in $o(s^{-L})$ and disregarded.
		After rearranging the others by gathering the terms with the same degree, one obtains the expression
		\begin{equation*}
			\sum_{l = 0}^{L} \biggl( \nabla^{F}_{V} \Sexp (D) + \sum_{n = 1}^{l} \FHWC^{(l)}_{V} \bigl( \Sexp[l-n] (D) \bigr) \biggr) s^{-l} + o(s^{-L}) \, .
		\end{equation*}
		By comparison with~\eqref{eq:formal_derivative}, this sum expresses the truncation of the formal Hitchin-Witten covariant derivative of $\mathcal{R}(D)$.
		The assertion follows.
	\end{proof}
	
	Using the explicit definition of $\FHWC$ in~\eqref{eq:formal_derivative}, the condition for the series to be covariantly constant becomes
	\begin{equation*}
		\sum_{l = 0}^{\infty} \dF \big( \Sexp (D) \big) s^{-l} = \rec{2k} \sum_{n=1}^{\infty} \sum_{l=0}^{\infty} \left( \frac{ik}{s} \right)^n \Big[ b - (-1)^{n} \overline{b} , \Sexp (D) \Big] s^{-l} \, .
	\end{equation*}
	By collecting the coefficients of each power of $\inv{s}$ on the right-hand side one can finally reduce this to
	\begin{equation}
		\label{eq:recursion}
		 \dF \Sexp[l] (D) = \rec{2k} \sum_{n = 1}^{l} \left( ik \right)^n \left[ b - (-1)^{n} \overline{b} , \Sexp[l - n] (D) \right] \, .
	\end{equation}
	This implies in particular that the right-hand side should be an exact form with respect to $\dF$, which gives a necessary condition on the solution of the first $l$ steps in order for the next one to exist.
	An obstruction comes then from the differential of the right-hand side; the next result, which is a re-phrasing of Theorem~\ref{thm:no_Sobstruction}, shows that this obstruction vanishes.
	
	\begin{Proposition}
		\label{prop:no_Sobstruction}
		Suppose that $l$ steps of the recursion have been solved, giving an $\Sexp[n] (D)$ for every $n \leq l$.
		Then
		\begin{equation*}
			\dF \Biggl( \sum_{n = 1}^{l} \left( ik \right)^n \left[ b - (-1)^{n} \overline{b} , \Sexp[l - n] (D) \right] \Biggr) = 0
		\end{equation*}
	\end{Proposition}
	
	\begin{proof}
		We proceed by using the properties of the exterior differential to expand
		\begin{equation}
			\label{eq:Sobstruction}
			\begin{gathered}
				\dF \Biggl( \sum_{n = 1}^{l} \left( ik \right)^n \! \left[ b - (-1)^n \overline{b} , \Sexp[l-n] (D) \right] \Biggr) =	\\
				= \! \sum_{n = 1}^{l} \! \left( ik \right)^n \! \biggl( \! \left[ \dF \! \left( b \!-\! (\!-1)^n \overline{b} \right) \!,\! \Sexp[l-n] (D) \right]
				\!\!+\!\! \left[ \! \left( b \!-\! (\!-1)^n \overline{b} \right) \! \wedge \! \dF \Sexp[l-n] (D)\right] \! \biggr) .
			\end{gathered}
		\end{equation}
		Using the recursive relation on the second term in the parentheses, replacing $j = m - n$ and then carefully exchanging the sums one obtains
		\begin{equation*}
			\begin{split}
				&\ \rec{2k} \sum_{n = 1}^{l} \sum_{j = 1}^{l-n} (ik)^{n+j} \Bigl[ \bigl( b - (-1)^n \overline{b} \bigr) \wedge \Bigl[ \bigl( b - (-1)^{j} \overline{b} \bigr) \wedge \Sexp[l-n-j] (D) \Bigl] \Bigr] = \\
				= &\ \rec{2k} \sum_{n = 1}^{l} \sum_{m = n+1}^{l} (ik)^{m} \Bigl[ \bigl( b - (-1)^{n} \overline{b} \bigr) \wedge \Bigl[ \bigl( b - (-1)^{m - n} \overline{b} \bigr) \wedge \Sexp[l - m] (D) \Bigr] \Bigr] = \\
				= &\ \rec{2k} \sum_{m = 2}^{l} \sum_{n = 1}^{m - 1} (ik)^{m} \Bigl[ \bigl( b - (-1)^{n} \overline{b} \bigr) \wedge \Bigl[ \bigl( b - (-1)^{m - n} \overline{b} \bigr) \wedge \Sexp[l - m] (D) \Bigr] \Bigr] \, .
			\end{split}
		\end{equation*}
		Except for the factor $(ik)^{m}$, each term of this sum can be expanded as
		\begin{equation}
			\label{eq:Sobstruction2nd}
			\begin{gathered}
				\Bigl[ b \wedge \Bigl[ b \wedge \Sexp[l-m] (D) \Bigr] \Bigr] + (-1)^{m} \Bigl[ \overline{b} \wedge \Bigl[ \overline{b} \wedge \Sexp[l-m] (D) \Bigr] \Bigr] + \\
				- ( -1)^{n} \Bigl( (-1)^{m} \Bigl[ b \wedge \Bigl[ \overline{b} \wedge \Sexp[l-m] (D) \Bigr] \Bigr] + \Bigl[ \overline{b} \wedge \Bigl[ b \wedge \Sexp[l-m] (D) \Bigr] \Bigr] \Bigr) \, .
			\end{gathered}
		\end{equation}
		The Jacobi identity of Lemma~\ref{lemma:twisteddiff} implies that
		\begin{equation*}
			\Bigl[ b \wedge \bigl[ b \wedge \Sexp[l-m] (D) \bigr] \Bigr] - \Bigl[ b \wedge \bigl[ \Sexp[l-m] (D) \wedge b \bigr] \Bigr] + \Bigl[ \Sexp[l-m] (D) \wedge \bigl[ b \wedge b \bigr] \Bigr] = 0 \, .
		\end{equation*}
		By the centrality of $[ b \wedge \overline{b}]$ and the commutation rules for these forms, this boils down to the vanishing of the left-most term, showing that the first term in~\eqref{eq:Sobstruction2nd} is $0$.
		It can be argued in the same way that the second term gives no contribution either.
		Due to the factor $(-1)^n$ in the remaining terms, the sum over $n$ yields $0$ whenever the summation range has even length, i.e. when $m$ is odd.
		On the other hand, for even $m$ the Jacobi identity gives
		\begin{equation*}
			\Bigl[ b \!\wedge\! \bigl[ \overline{b} \!\wedge\! \Sexp[l-m] (D) \bigr] \Bigr] \!+\! \Bigl[ \overline{b} \!\wedge\! \bigl[ b \wedge \Sexp[l-m] (D) \bigr] \Bigr] \!=\! - \! \Bigl[ \bigl[ b \wedge \overline{b} \bigr] \wedge \Sexp[ l - m ] (D) \Bigr] .
		\end{equation*}
		From this we conclude that the second part ot the sum in~\eqref{eq:Sobstruction} equals
		\begin{equation}
			\label{eq:Sobstruction_part2}
			- \rec{2k} \sum_{0 < 2r \leq l} (ik)^{2r} \Bigl[ \bigl[ b \wedge \overline{b} \bigr] \wedge \Sexp[ l - 2r ] (D) \Bigr] \, .
		\end{equation}
		
		On the other hand, the first part of~\eqref{eq:Sobstruction} can be handled by comparison with the equation expressing the flatness of $\FHWC$.
		Applying the curvature to $\Sexp[ l -n ] (D)$ and isolating the term in $s^{-n}$ gives
		\begin{equation*}
			\biggl[ d_{\Teich}^{F} ( b - \! (-1)^n \overline{b} ) \!+\! \rec{4k} \sum_{m = 1}^{n-1} \! \left[ \left( b - (-1)^{m} \overline{b} \right) \!\wedge\! \left( b - (-1)^{n-m} \overline{b} \right) \right] , \Sexp[ l - n ] (D) \biggr] = 0 \, .
		\end{equation*}
		Using the centrality of $[b \wedge b]$ and $[\overline{b} \wedge \overline{b}]$ one obtains that the first part of~\eqref{eq:Sobstruction} is
		\begin{equation*}
			\rec{4k} \sum_{n = 1}^{l} \sum_{m = 1}^{n-1} (ik)^{n} (-1)^{m} \left( 1 + (-1)^{n} \right) \left[ \left[ b \wedge \overline{b} \right] , \Sexp[ l - n ] (D) \right] \, .
		\end{equation*}
		As before, the sum over $m$ gives $0$ whenever $n-1$ is even, leaving
		\begin{equation*}
			\rec{2k} \sum_{0 < 2r \leq l} (ik)^{2r} \Bigl[ \bigl[ b \wedge \overline{b} \bigr] , \Sexp[ l - n ] (D) \Bigr] \, .
		\end{equation*}
		The proof is concluded by comparing with~\eqref{eq:Sobstruction_part2}.
	\end{proof}
	
	As a side remark we notice that, if $D$ is a differential operator of order $n$, then $[b \pm \overline{b} , D]$ takes values in operator of order $n+1$, generically.
	This justifies our claim that an asymptotic expansion for curve operators with coefficients of the same kind need not exist, and one should look for $\Sexp(f)$ as a differential operator of order $l$.
	
	\subsection{First step of the recursion for curve operators}
	
		For $l = 1$ and $D = \curveop$, Equation~\eqref{eq:recursion} reads
		\begin{equation}
			\label{eq:1st_recursion}
			\dF \Sexp[1](f) = \frac{i}{2} \bigl[ b + \overline{b} , \curveop \bigr] \, .
		\end{equation}
		Identifying a function with its curve operator for notational convenience, and recalling the definition of $b$ and $\overline{b}$, one has
		\begin{equation*}
			b + \overline{b} = \Delta_{\tilde{G}} + 2 \nabla_{\tilde{G} \cdot \dif F} - 2 \lambda \dT F \, .
		\end{equation*}
		Notice that the first and last terms are both exact, with primitives $- \Delta$ and $-2 \lambda F$, respectively.
		However, the last term does not contribute to the commutator in~\eqref{eq:1st_recursion}, which becomes
		\begin{equation*}
			\frac{i}{2} \bigl[ b + \overline{b} , \curveop \bigr] = \Delta_{\tilde{G}} f + 2 \nabla_{\tilde{G} \cdot \dif f} + 2 \dif f \contr \tilde{G} \contr \dif F \, .
		\end{equation*}
		By our assumption $\dF \curveop = 0$, $f$ does not depend on $\sigma$, so the first term is clearly $\dT$-exact, with primitive $- \Delta f$; being central as a differential operator, this is also a primitive for $\dF$.
		The second term can be written as $- 2 \dT \nabla_{\tilde{g} \cdot \dif f}$, while on the other hand
		\begin{equation*}
			- 2 \bigl[ \dT F , \nabla_{\tilde{g} \cdot \dif f} \bigr] = 2 \dif f \contr \tilde{g} \contr \dif \dT F = 2 \dT \bigl( \dif f \contr \tilde{g} \contr \dif F \bigr) + 2 \dif f \contr \tilde{G} \contr \dif F \, .
		\end{equation*}
		This way one obtains the missing term, up to an exact correction.
		All in all, we have found that
		\begin{equation*}
			- \dF \bigl( 2 \nabla_{\tilde{g} \cdot \dif f} + 2 \dif f \contr \tilde{g} \contr \dif F + \Delta f \bigr) = \bigl[ \Delta_{\tilde{G}} + 2 \nabla_{\tilde{G} \cdot \dif F} - 2 \lambda \dT F , f \bigr] \, ,
		\end{equation*}
		thus proving Theorem~\ref{thm:curveop} for
		\begin{equation}
			\label{eq:1st_solution}
			\Sexp[1] (f) = - \frac{i}{2} \bigl( 2 \nabla_{\tilde{g} \cdot \dif f} + 2 \curveop[\dif f \cdot \tilde{g} \cdot \dif F + \Delta f] \bigr) \, .
		\end{equation}
	
	\subsection{\texorpdfstring{Recursion in $t$ and $\overline{t}$ and obstruction}{Recursion in t and bar t and obstruction}}
		\label{sec:t_obstruction}
		
		We include this section to show where the complications arise in the analogous formal approach in the full parameter $t$, in addition to the technical ones discussed in Section~\ref{sec:convergence}.
		
		We now consider $\EHWC$ as a formal connection, being manifestly a Laurent polynomial in $t$ and $\overline{t}$.
		We address the problem of finding a perturbation $\Triv (f)$ of $\curveop$ in such a way that $\EHWC \Triv(f) \equiv 0$ formally.
		Explicitly, for every vector $V$ on $\Teich$ the equation we are interested in reads
		\begin{equation*}
			\sum_{l = 0}^{\infty} V \left[ \triv (f) \right] + \sum_{l =0}^{\infty} \left[ \curveop[{V[F]}] , \triv (f) \right] + \sum_{l = 0}^{\infty} \left[ \rec{2t} b(V) - \rec{2 \overline{t}} \overline{b} (V) , \triv (f) \right]= 0.
		\end{equation*}
		Putting for convenience $\triv[-1] (f) = 0$, by separating this degree-by-degree one obtains for every $l \geq 0$ the recursive relation
		\begin{equation}
			\label{eq:formal_t_recursion}
			\dF \triv[l+1] (f) = - \Bigl[ \rec{2t} b - \rec{2 \overline{t}} \overline{b} , \triv (f) \Bigr] \, .
		\end{equation}
		Again, a necessary condition for the existence of a solution $\triv[l+1] (f)$ is for right-hand side to be $\dF$-closed.
		However, its differential is given by the following formula.
	
		\begin{Proposition}
			If the recursive relation is satisfied for $0 = \triv[-1] (f)$, $\curveop = \triv[0] (f), \dots, \triv[l] (f)$, then the differential of the right-hand side of~\eqref{eq:formal_t_recursion} is
			\begin{equation}
				\label{eq:obstruction}
				\dF \biggl[ \rec{2t} b - \rec{2\overline t} \overline b , \triv[l] (f) \biggr] = \rec{4 \abs t^2} \Bigl[ \bigl[ b \wedge \overline b \bigr] , \triv[l] (f) \Bigr] \, .
			\end{equation}
		\end{Proposition}
		
		\begin{proof}
			Let for convenience
			\begin{equation*}
				\tilde{b} = \rec{2t} b - \rec{2\overline t} \overline b \, .
			\end{equation*}
			Following the lines of the proof of Theorem~\ref{thm:no_Sobstruction}, we proceed by direct computation of the twisted differential and obtain
			\begin{equation}
				\label{eq:expldiff}
				\dF \Bigl[ \tilde{b} , \triv[l] (f) \Bigr] = \Bigl[ \bigr(\dF \tilde{b} \bigr) \wedge \triv[l] (f) \Bigr] - \Bigl[ \tilde{b} \wedge \dF \triv[l] (f) \Bigr] \, .
			\end{equation}
			Using the recursive relation, the second term can be written as
			\begin{equation*}
				- \Bigl[ \tilde{b} \wedge \dF \triv[l] (f) \Bigr] =
				\Bigl[ \tilde{b}  \wedge \bigl[ \tilde{b} \wedge \triv[l-1] (f) \bigr] \Bigr] \, .
			\end{equation*}
			Combining this with the Jacobi identity we find
			\begin{equation*}
				- \Bigl[ \tilde{b} \wedge \dF \triv[l] (f) \Bigr] = \rec{2} \Bigl[ \bigl[ \tilde{b} \wedge \tilde{b} \bigr] , \triv[l-1] (f) \Bigr] \, .
			\end{equation*}
			By expanding $\tilde{b}$ and using to the centrality of $[b \wedge b]$ and $[\overline{b} \wedge \overline{b}]$, one finds
			 \begin{equation*}
				- \Bigl[ \tilde{b} \wedge \dF \triv[l] (f) \Bigr] = - \rec{4 \abs{t}^2} \Bigl[ \bigl[ b \wedge \overline{b} \bigr] \wedge \triv[l-1] (f) \Bigr] \, .
			\end{equation*}
			By induction, for positive $l$ this represents the obstruction to the existence of $\triv[l] (f)$, so it vanishes, while for $l = 0$ one has simply $\triv[-1] (f) = 0$.
			This proves that the second term on the right-hand side of~\eqref{eq:expldiff} vanishes.
			As for the other term, we shall compute it by comparison with the explicit expression for the curvature of $\EHWC$ applied to $\triv (f)$.
			The flatness of $\nabla^{F}$ can be expressed as
			\begin{equation*}
				0 = \Bigl[ \dF \tilde{b} + \rec{2} \bigl[ \tilde{b} \wedge \tilde{b} \bigr] , \triv[l] (f) \Bigr] \, .
			\end{equation*}
			From this we conclude that
			\begin{equation*}
				\dF \bigl[ \tilde{b} , \triv[l] (f) \bigr] = - \rec{2} \Bigl[ \bigl[ \tilde{b} \wedge \tilde{b} \bigr] \wedge \triv[l] (f) \Bigr] \, ,
			\end{equation*}
			and the result follows by expanding $\tilde{b}$ again.
		\end{proof}
	
		The right-hand side of~\eqref{eq:obstruction} does not vanish in general.
		Indeed, for $V$ and $W$ vector fields on $\Teich$, using the symbols of $\bigl[ b \wedge \overline{b} \bigr] (V,W)$ computed in~\cite{AG14} (Proposition 4.9) one can see that
		\begin{equation*}
				\sigma_{1} \bigl[ [ b (V) , \overline{b} (W) ] , \curveop \bigr] = - 8 i k \dif f \contr \Theta(V,W) = 8 \pi k \dif f \contr \bigl( \dT G \bigr) (V,W) \, .
		\end{equation*}
		This shows that the recursion~\eqref{eq:formal_t_recursion} has in general no solution, even at the first step $l = 0$, and this approach fails even for curve operators.
	
\section{Solution of the recursion in genus 1}
	\label{sec:genus1}

	We now consider the situation where $\ModSp = \ModSp^{n,0}$ is the moduli space of flat $\SU(n)$-connections over a closed, oriented, smooth surface $\Sigma$ of genus $1$.
	
	\subsection{The Hitchin-Witten connection}
	
	Although some of the conditions of Section~\ref{sec:AG14} are not met in the situation at hand, the construction of the Hitchin-Witten connection can still be carried out.
	In $2\pi$-periodic coordinates on $\Sigma$, any gauge orbit of flat $\SU(n)$-connections or isotopy class of Riemann surface structures has a representative with constant coefficients.
	The tangent space of $\ModSp$ at any smooth point can by identified with that of $1$-forms on $\Sigma$ with constant coefficients in the Cartan sub-algebra of diagonal matrices in $\su(n)$.
	Under this correspondence, all the tensor fields introduced in Section~\ref{sec:AG14} are represented by matrices whose entries depend on $\sigma \in \Teich$ alone.
	In particular, the metric $g$ has constant entries in these coordinates, so the Levi-Civita connection is the trivial one and hence flat; since $\tilde{G} (V)$ has constant coefficients for any vector $V$ on $\Teich$, it is also parallel, so one may write
	\begin{equation*}
		\nabla \tilde{G} = 0 \, .
	\end{equation*}
	It can be checked that the family of complex structures is holomorphic, rigidity following from the equation above.
	Moreover, we have as Ricci potential $F = 0$ with $\lambda = 0$, so one has simply
	\begin{equation*}
		b = \Delta_{G} = \nabla^{2}_{G}
		\qquad \text{and} \qquad
		\overline{b} = \Delta_{\overline{G}} = \nabla^{2}_{\overline{G}} \, ,
	\end{equation*}
	and $\dF$ can be replaced with the usual differential $\dT$.
	The Hitchin-Witten connection can be defined in this case in the same way as in Section~\ref{sec:AG14}.
	
	\begin{Proposition}[Witten]
		\label{prop:trivialisation}
		The Hitchin-Witten connection for a surface of genus $1$ is gauge-equivalent to the trivial connection, and the equivalence is realised as
		\begin{equation*}
			\exp \bigl( r \Delta \bigr) \HWC \exp \bigl( - r \Delta \bigr) = \nablatr \, ,
		\end{equation*}
		where $r \in \Cx$ is such that
		\begin{equation}
			\label{eq:defr}
			e^{4rk} = - \frac{\overline{t}}{t} \, .
		\end{equation}
	\end{Proposition}
	
	\begin{proof}
		In order to prove the statement it is enough to show that
		\begin{equation*}
			\exp(-r \Delta) \dT \bigl[ \exp (r \Delta) \bigr] = \rec{2t} b - \rec{2 \overline{t}} \overline{b} \, .
		\end{equation*}
		We will proceed by differentiating the exponential series of $r \Delta$ term-wise, using the fact that $\dT \Delta = - (b + \overline{b})$.
		To this end, we compute the commutator of $b \pm \overline{b}$ with the powers of $\Delta$: using the flatness of $\ModSp$ and the parallelism of $\tilde{g}$ and $\tilde{G}$, one obtains by direct computation that
		\begin{equation*}
			\bigl[ b \pm \overline{b} , \Delta \bigr] = 4 k \bigl( b \mp \overline{b} \bigr)
		\end{equation*}
		and, by induction,
		\begin{equation*}
			\bigl[ b \pm \overline{b} , \Delta^{n} \bigr] = \sum_{l = 1}^{n} \binom{n}{l} (4k)^{l} \Delta^{n-l} \Bigl( b \pm (-1)^{l} \overline{b} \Bigr) \, .
		\end{equation*}
		Using this we can compute
		\begin{equation*}
			\begin{split}
				\dT \bigl(\Delta^{n} \bigr) ={}& - \sum_{j = 1}^{n} \Delta^{n-j} \bigl( b + \overline{b} \bigr) \Delta^{j-1} = \\
					={}& - n \Delta^{n-1} \bigl( b + \overline{b} \bigr) - \sum_{j = 1}^{n} \sum_{l = 1}^{j-1} \binom{j-1}{l} (4k)^{l} \Delta^{n - l - 1} \bigl( b + (-1)^{l} \overline{b} \bigr) \, .
			\end{split}
		\end{equation*}
		One can now exchange the sums and use the identity
		\begin{equation*}
			\sum_{j = l+1}^{n} \binom{j-1}{l} = \binom{n}{l+1}
		\end{equation*}
		to find
		\begin{equation*}
			\dT \bigl(\Delta^{n} \bigr) = - \sum_{l = 0}^{n-1} \binom{n}{l+1} (4k)^{l} \Delta^{n - l - 1} \bigl( b + (-1)^{l} \overline{b} \bigr) \, .
		\end{equation*}
		We now apply this to the derivative of the exponential series
		\begin{equation*}
			\dT \exp \bigl( r \Delta \bigr) = - \sum_{n = 1}^{\infty} \sum_{l=0}^{n-1} \binom{n}{l+1} \frac{ (4k)^{l} r^{n}}{n!} \Delta^{n - l - 1} \bigl( b + (-1)^{l} \overline{b} \bigr) \, .
		\end{equation*}
		We now change $l$ with $n - l - 1$, switch the sums and further change $n$ with $n+l$ to find
		\begin{equation*}
			\begin{split}
				\dT \exp \bigl( r \Delta \bigr) ={}& - \sum_{l=0}^{\infty} \sum_{n=1}^{\infty} \rec{n!} \rec{l!} (4k)^{n-1} r^{n+l} \Delta^{l} \bigl( b + (-1)^{n-1} \overline{b} \bigr) = \\
					={}& - \rec{4k} \sum_{l=0}^{\infty} \frac{(r \Delta)^{l}}{l!} \Biggl( \biggl( \sum_{n=1}^{\infty} \frac{(4kr)^{n}}{n!} \biggr) b - \biggl( \sum_{n=1}^{\infty} \frac{(-4kr)^{n}}{n!} \biggr) \overline{b} \Biggr) \, .
			\end{split}
		\end{equation*}
		Recognising the sums next to $b$ and $\overline{b}$ as truncated exponential series, they converge to $e^{\pm 4kr} - 1$, and using the condition~\eqref{eq:defr} we finally find
		\begin{equation*}
			\dT \exp \bigl( r \Delta \bigr) = - \rec{4k} \exp \bigl( r \Delta \bigr) \biggl( \frac{2k}{t} b - \frac{2k}{\overline{t}} \overline{b} \biggr) = \exp \bigl( r \Delta \bigr) \biggl( \rec{2t} b - \rec{2 \overline{t}} \overline{b} \biggr) \, .
		\end{equation*}
		This concludes the proof.
	\end{proof}
	
	\subsection{The formal connection and the recursion}
	
	We stress that the proofs of Theorem~\ref{thm:formal_connection} and Lemma~\ref{lemma:asymptotic} use none of the specific hypotheses of Section~\ref{sec:AG14} directly, but they do use the projective flatness of $\HWC$.
	Since we have proven that, in the case at hand, the Hitchin-Witten connection is actually flat, those two results extend to this situation, thus giving a formal Hitchin-Witten connection.
	It makes then sense to look for a formal trivialisation as a solution to a recursive system of differential equations as in~\eqref{eq:recursion}.
	Because in this situation $b + \overline{b} = \Delta_{\tilde{G}} = - \dT \Delta$ is exact, a solution to the first step of the recursion can be found as
	\begin{equation*}
		\dT \Bigl[ - \frac{i}{2} \Delta , D \Bigr] = \frac{i}{2} \bigl[ b + \overline{b} , D \bigr] \, ,
	\end{equation*}
	which for curve operators agrees indeed with the solution $\Sexp[1] (f)$ found in~\eqref{eq:1st_solution}, since
	\begin{equation*}
		\Bigl[ - \frac{i}{2} \Delta , \curveop \Bigr] = - \frac{i}{2} \bigl( \curveop[\Delta f] + 2 \nabla_{\tilde{g} \cdot \dif f} \bigr) \, .
	\end{equation*}
	This hints that one may look for a solution of the next steps by iterating the commutator with this operator.
	Since $\Delta$ is defined directly from the Kähler metric, it follows that a solution constructed this way is automatically invariant under the action of $\Mod_{\Sigma}$.
	
	\begin{Definition}
		We let
		\begin{equation*}
			a := - \frac{i}{2} \Delta \, ,
		\end{equation*}
		and for every non-negative integer $l$ and $\sigma$-independent $D \in \DiffD$ we set
		\begin{equation*}
			\Pexp (D) := \frac{\ad_{a}^{l} (D)}{l!} \, .
		\end{equation*}
	\end{Definition}
	
	When $D = \curveop$ for some function $f$, we shall use the short-hand $\Pexp (f)$ in place of $\Pexp (\curveop)$.
	In this notation, one has that $\Sexp[1] (f) = \Pexp[1] (f)$.
	With a slight modification to the calculations in the proof of Proposition~\ref{prop:trivialisation}, one can verify that
	\begin{equation*}
		\dT \ad_{a}^l = \frac{1}{4k} \sum_{n=1}^{l} \binom{l}{n} (2ik)^n \ad_{b - (-1)^n \bar{b}} \ad_{a}^{l-n} \, .
	\end{equation*}
	After introducing the relevant factorials, one obtains the following statement.
	
	\begin{Proposition}
		\label{prop:Adiff}
		The operators $\Pexp (D)$ satisfy the relation
		\begin{equation*}
			\dF \Pexp (D) = \sum_{n=1}^{l} \frac{(2ik)^n}{4kn!} \left[ b - (-1)^n \bar{b} , \Pexp[l-n] (D) \right] \, .
		\end{equation*}
	\end{Proposition}
	
	This relation is identical to the recursion~\eqref{eq:recursion}, except for the coefficients.
	For this reason, it is natural to look for $\Sexp$ as a linear combination of the $\Pexp[n]$'s
	\begin{equation}
		\label{eq:ansatz}
		\Sexp = \sum_{r = 0}^{l} \Coeff{r}{l} \Pexp[l-r] \, ,
	\end{equation}
	where we have simplified the notation by leaving $D$ unspecified in $\Sexp$ and $\Pexp$.
	Here the coefficients $\Coeff{r}{l}$ are constant complex parameters, and the condition $\Sexp[0] = f$ is equivalent to $\Coeff{0}{0} = 1$.
	After substituting~\eqref{eq:ansatz} and manipulating the sums, the recursion~\eqref{eq:recursion} reads
	\begin{equation*}
		\begin{split}
				&\sum_{n=1}^{l} \sum_{r=0}^{l-n} \frac{(ik)^n}{2k} \Coeff{r}{l-n} \Bigl[ b - (-1)^n \bar{b} , \Pexp[l-n-r] \Bigr] = \\
			={}&\sum_{r=0}^{l-1} \sum_{n = 1}^{l-r} \frac{(ik)^n}{2k} \Coeff{r}{l-n} \Bigl[ b - (-1)^n \bar{b} , \Pexp[l-n-r] \Bigr] = \\
			={}& \sum_{r=0}^{l} \sum_{m=r+1}^{l} \frac{(ik)^{m-r}}{2k} \Coeff{r}{l-m+r} \Bigl[ b - (-1)^{m-r} \bar{b} , \Pexp[l-m] \Bigr] = \\
			={}& \sum_{m=1}^{l} \biggl( \sum_{r=0}^{m-1} \frac{(ik)^{m-r}}{2k} \Coeff{r}{l-m+r} \biggr) \Bigl[ b - (-1)^{m-r} \bar{b} , \Pexp[l-m] \Bigr] \, .
		\end{split}
	\end{equation*}
	Applying the same substitution and Proposition~\ref{prop:Adiff} to the left-hand side yields
	\begin{equation*}
		\begin{split}
			\dT \Sexp ={}& \sum_{r=0}^{l} \Coeff{r}{l} \sum_{n=1}^{l-r} \frac{(2ik)^n}{4kn!} \Bigl[ b - (-1)^n \bar{b} , \Pexp[l-r-n] \Bigr] = \\
			={}& \sum_{r=0}^{l} \sum_{m=r+1}^{l} \frac{(2ik)^{m-r}}{4k(m-r)!} \Coeff{r}{l} \Bigl[ b - (-1)^{m-r} \bar{b} , \Pexp[l-m] \Bigr] = \\
			={}& \sum_{m=1}^{l} \biggl( \sum_{r=0}^{m-1} \frac{(2ik)^{m-r}}{4k(m-r)!} \Coeff{r}{l} \biggr) \Bigl[ b - (-1)^{m-r} \bar{b} , \Pexp[l-m] \Bigr] \, .
		\end{split}
	\end{equation*}
	After separating holomorphic and anti-holomorphic parts, the recursion translates then into the system
	\begin{equation*}
		\begin{gathered}
			\sum_{m=1}^{l} \biggl( \sum_{r=0}^{m-1} \frac{(ik)^{m-r}}{2k} \Coeff{r}{l-m+r} - \sum_{r=0}^{m-1} \frac{(2ik)^{m-r}}{4k(m-r)!} \Coeff{r}{l} \biggr) \Bigl[ b , \Pexp[l-m] \Bigr] = 0 \, , \\
			\sum_{m=1}^{l} \biggl( \sum_{r=0}^{m-1} \frac{(-ik)^{m-r}}{2k} \Coeff{r}{l-m+r} - \sum_{r=0}^{m-1} \frac{(-2ik)^{m-r}}{4k(m-r)!} \Coeff{r}{l} \biggr) \Bigl[ \bar{b} , \Pexp[l-m] \Bigr] = 0 \, .
		\end{gathered}
	\end{equation*}
	One may now observe by induction that, in the case of a curve operator $\curveop$, the top symbol of $\Pexp (f)$ essentially consists of the $l$-th derivatives of $f$.
	Generically, the $\Pexp (f)$'s form a family of differential operators of increasing order, hence linearly independent.
	Therefore, the above equations hold for every differential operator $D$ if and only if each of the summands in $m$ vanishes.
	To summarise, we have established the following fact.
	
	\begin{Proposition}
		Assuming that each $\Sexp$ is of the form~\eqref{eq:ansatz}, where $\Coeff{r}{l}$ is independent of $D$, the recursion~\eqref{eq:recursion} is equivalent to the system of equations
		\begin{equation*}
			E_{m,l}^{\pm} = 0
		\end{equation*}
		for all pairs of positive integers $m \leq l$, where
		\begin{equation*}
			E_{m,l}^{\pm} = \sum_{r=0}^{m-1} \left( \frac{(\pm 2ik)^{m-r}}{2(m-r)!} \Coeff{r}{l} - (\pm ik)^{m-r} \Coeff{r}{l-m+r} \right) \, .
		\end{equation*}
	\end{Proposition}
	
	We now study the equivalent system
	\begin{equation}
		\label{eq:numrec}
			E^{\pm}_{m,l} \mp ik E^{\pm}_{m-1 , l-1} = 0		\qquad	1 \leq m \leq l \, ,
	\end{equation}
	where it is understood that $E^{\pm}_{0,l} = 0$ for all $l$.
	This way, all the terms with $\Coeff{r}{l-m+r}$ disappear, except for the one corresponding to $r = m-1$.
	It is now convenient to replace $r$ by  $\rho - 1$ and collect the coefficients with fixed $l$ into the vectors
	\begin{equation*}
		\Xoeff{}{l} = \bigl( \Coeff{\rho-1}{l} \bigr)_{1 \leq \rho \leq l+1} \, ,
			\qquad
		\Troeff{}{l} = \bigl(\Coeff{\rho-1}{l} \bigr)_{1 \leq \rho \leq l} \, .
	\end{equation*}
	In this notation, the equation reads
	\begin{equation*}
		\sum_{\rho = 1}^{m} \frac{(\pm 2ik)^{m-\rho+1}}{2 (m-\rho+1)!} \Xoeff{\rho}{l} = \pm ik \sum_{\rho = 1}^{m-1} \frac{(\pm 2ik)^{m-\rho}}{2 (m-\rho)!} \Xoeff{\rho}{l-1} \pm ik \Xoeff{m}{l-1} \, .
	\end{equation*}
	Notice that, as $\rho$ ranges between $1$ and $l$, the coefficients $\Xoeff{\rho}{l}$ are the entries of the vector $\Troeff{}{l}$, while the $\Xoeff{\rho}{l-1}$'s are those of $\Xoeff{}{l}$.
	The equation may then be seen as a linear relation
	\begin{equation*}
		L^{(l)}_{\pm} \Troeff{}{l} = R^{(l)}_{\pm} \Xoeff{}{l-1} \, ,
	\end{equation*}
	where $L^{(l)}_{\pm}$ and $R^{(l)}_{\pm}$ are square matrices of size $l$.
	Notice that the entry in position $(m,\rho)$ in each of these vanishes whenever $\rho > m$, which makes them lower-triangular.
	Moreover, each entry depends only on the difference $m - \rho$, which means that they are polynomials in the standard nilpotent matrix
	\begin{equation*}
		N
		= \begin{pmatrix}
			0 	&				&				&				&\\
			1	&	0			&				&				&	\\
				&	\ddots	&	\ddots	&				&	\\
				&				&	1			&	0			&	\\
				&				&				&	1			&	0
		\end{pmatrix} \, .
	\end{equation*}
	More precisely, one may write
	\begin{equation*}
		L^{(l)}_{\pm} = \sum_{n=0}^\infty \frac{(\pm 2i k)^{n + 1}}{2 (n + 1)!} N^n \, ,
			\qquad
		R^{(l)}_{\pm} = \pm ik \1 \pm i k \sum_{n=1}^\infty \frac{(\pm 2i k)^{n}}{2 n!} N^n \, .
	\end{equation*}
	Both matrices are invertible, being triangular with no zeroes on the diagonal; in particular, the inverse of $L^{(l)}_{\pm}$ determines $\Troeff{}{l}$ in terms of $\Xoeff{}{l-1}$.
	
	As is easily seen, the sums expressing the matrices are in fact the Taylor series of the analytic functions
	\begin{equation*}
		\begin{gathered}
			\sum_{n=0}^\infty \frac{(\pm 2i k)^{n + 1}}{2 (n + 1)!} z^n = \frac{e^{\pm 2ik z} - 1}{2 z} \, , \\[.4em]
			\pm ik \biggl( 1 + \sum_{n=1}^\infty \frac{(\pm 2i k)^{n}}{2 n!} z^n \biggr) = \pm ik \frac{e^{2ikz}+1}{2} \, .
		\end{gathered}
	\end{equation*}
	Since the Taylor series of a product is the formal product of the Taylor series of the factors, one may turn the problem from matrices to holomorphic functions.
	Indeed, if $\varphi$ is a holomorphic function in a neighbourhood of $0 \in \Cx$, it makes sense to write $\varphi(N)$ to mean the evaluation at $N$ of the Taylor series of $\varphi$ at $0$, because the sum terminates.
	Therefore, one may finally conclude that the system of equations~\eqref{eq:numrec} for even $m$ is equivalent to
	\begin{equation*}
		\Troeff{}{l} = \pm ik \frac{ e^{\pm 2 i k N} + 1}{e^{\pm 2 i k N} -1} N \Xoeff{}{l-1} = \pm ik \tanh (\pm ikN) N \Xoeff{}{l-1} \, .
	\end{equation*}
	Since the hyperbolic tangent is an odd function, the signs may finally be disregarded: the two sets of equations corresponding to the signs are in fact equivalent.
	
	This essentially concludes the discussion of the solutions of the numeric recursion, which is summarised as the following reformulation of the first part of Theorem~\ref{thm:introtrivialisation}.
	
	\begin{Theorem}
		There exists a unique solution of the numeric recursion for every choice of the coefficients $\Coeff{l}{l}$ for $l \geq 1$. 
	\end{Theorem}
	
	\begin{proof}
		We have proved that the system is equivalent to the vanishing of $E^{\pm}_{m,l} \mp ik E^{\pm}_{m-1,l-1} = 0$ for every positive odd $m \leq l$.
		In turn, this set of equations is equivalent to a triangular one expressing $\Troeff{}{l}$ in terms of $\Xoeff{}{l-1}$.
		This leaves as free variables precisely the last entry of each $\Xoeff{}{l}$, which corresponds to $\Coeff{l}{l}$ as claimed.
	\end{proof}
	
	In spite of this ambiguity, it should be noted that the solution is essentially unique.
	Indeed, consider the particular trivialisation $\mathcal{R}_{0}$ corresponding to $\Coeff{l}{l} = 0$ for every $l > 0$, and the $\mathcal{R}_{l}$ with $\Coeff{l}{l} = 1$ for one specific value of $l > 0$.
	By linearity, the difference $\mathcal{R}_{0} - \mathcal{R}_{l}$ is a solution of the recursion starting with $\Coeff{0}{0} = 0$ instead of $1$, and is therefore determined by the same expressions.
	It is then immediate from the explicit form of the solutions that this difference equals $s^{-l} \mathcal{R}_{0}$, and by linearity one obtains the following statement.
	
	\begin{Proposition}
		\label{prop:formal_solution_uniqueness}
		The solution $\mathcal{R}_{\alpha}$ corresponding to a sequence of coefficients $\Coeff{l}{l}$ is related to $\mathcal{R}_{0}$ as above by the relation
		\begin{equation*}
			\mathcal{R}_{\alpha} = \biggl( \sum_{l = 0}^{\infty} \Coeff{l}{l} s^{-l} \biggr) \cdot \mathcal{R}_{0} \, .
		\end{equation*}
	\end{Proposition}
	
	In other words, the various solutions differ by a factor of an invertible power series in $\inv{s}$.
	
	\subsection{Example: a solution from the trivialisation of \texorpdfstring{$\tilde{\Nabla}$}{HWC}}
	
	We conclude this work by showing one particular solution of the numeric recursion above, obtained from the trivialisation of $\HWC$ of Proposition~\ref{prop:trivialisation}.
	In fact, the key point of that result is the relation
	\begin{equation}
		\label{eq:formalHWCtriv}
		\dT \sum_{n = 0}^{\infty} \frac{(r \Delta)^{n}}{n!} = \sum_{n = 0}^{\infty} \frac{(r \Delta)^{n}}{n!} \Bigl( \rec{2t} b - \rec{2 \overline{t}} \overline{b} \Bigr) \, ,
	\end{equation}
	which holds formally in $\DiffD[[r]]$ by regarding $1/t$ and $1/\overline{t}$ as their corresponding power series.
	On the other hand, $r$ admits a Taylor series in $\inv{s}$ for $s \to \infty$, for which moreover the term of degree $0$ vanishes.
	Indeed, $\overline{t}/t$ is a unit complex number which only takes the value $-1$ when $s = 0$, while in the limit for $s \to \infty$ with $k$ fixed it goes to $1$.
	Therefore, one may use the branch of the logarithm on $\Cx \setminus \R_{\leq 0}$ with $\log(1) = 0$ to express $r$ as
	\begin{equation*}
		r = \rec{4k} \log \Bigl( - \frac{k - is}{k + is} \Bigr) \, .
	\end{equation*}
	This can indeed be expanded as a formal power series in $\inv{s}$, and the vanishing of the zero-order term follows from that, for $s \to \infty$, one has $r \to 0$.
	Explicitly, this series reads
	\begin{equation*}
		r = \rec{2k} \sum_{n = 0}^{\infty} \frac{(ik)^{2n+1}}{2n+1} s^{-2n-1} \, .
	\end{equation*}
	In particular, it makes sense to substitute $r$ in~\eqref{eq:formalHWCtriv} with this series, thus obtaining a relation in $\DiffA$.
	This allows one to regard $\exp(r \Delta)$ as an element of this algebra, and shows that for a $\sigma$-independent operator $D \in \DiffD$ one finally has
	\begin{equation*}
		\FHWC \Bigl( \exp(r \Delta) D \exp(-r \Delta) \Bigr) = 0 \, .
	\end{equation*}
	The final point to notice is that this parallel series is in fact of the form of~\eqref{eq:ansatz}, as can be seen by using the formal analogue of the Baker-Campbell-Hausdorff formula, which reads
	\begin{equation*}
		\begin{split}
			\Biggl( \sum_{n=0}^{\infty} \frac{ (r \Delta)^{n}}{n!} \Biggr) D \Biggl( \sum_{n=0}^{\infty} \frac{ (-r\Delta)^{n}}{n!} \Biggr) ={}& \sum_{n=0}^{\infty} \frac{ (- r \ad_{\Delta})^{n}(D)}{n!} \, \\
			={} & \sum_{n = 0}^{\infty} {\Biggl( \sum_{m = 0}^{\infty} \frac{(ik)^{2m}}{2 m + 1} s^{-2m-1} \Biggr)\!}^{n} \Pexp[n] (D) \, .
		\end{split}
	\end{equation*}
	Moreover, $\Coeff{r}{l}$ occurs in this expression as the coefficient of $\Pexp[l-r] (D) s^{-l}$, and it appears when expanding the summand with $n = l-r$.
	Applying this in particular to $r = l$ shows that this is the solution corresponding to $\Coeff{l}{l} = 0$ for every $l > 0$.
	This finally concludes the proof of Theorem~\ref{thm:introtrivialisation}.
		
\appendix

\section{Miscellanea on differential operators}
\subsection*{Symbols and useful algebraic relations}
\label{sec:symbols}

We shall present here the notion of the symbols of finite-order differential operators acting on sections of $\Lqnt^{k}$ over $\ModSp$.
Although we do not use them extensively for the computations in this work, they do play a role in making sense of the asymptotic convergence of our expansions for $t \to \infty$.
As this construction depends on the Riemannian metric on $\ModSp$, and hence on $\sigma \in \Teich$, we assume this parameter to be fixed throughout the construction.

Given $n \in \Z_{\geq 0}$ and a differential operator $D$ of finite order at most $n$, the $n$-th symbol $\sigma_n (D)$ of $D$ is defined \cite{Wel08} as a complex, totally symmetric, $n$-contra-variant tensor field.
It can be described in coordinates as the collection of the coefficients of the derivatives of order $n$ appearing in $D$, and it vanishes if and only if the order of $D$ is strictly smaller than $n$.
While in general there is no sensible way of defining a tensor out of the lower-order coefficients, this can be done using the connection on $\Lqnt$ and the Levi-Civita connection on $\ModSp$.

Let $\psi$ be a section of $\Lqnt^{k}$: then $\nabla \psi$ is a section of $\Tan^{*}\!\ModSp \otimes \Lqnt^{k}$, and by iteration one can obtain a section $\nabla^n \psi$ of $(\Tan^{*} \! \ModSp)^{\otimes n} \otimes \Lqnt^{k}$.
Given an $n$-contra-variant tensor field $T$ on $\ModSp$, it makes then sense to contract it with the $n$ indices of $\nabla^n \psi$, thus obtaining a new section of $\Lqnt^{k}$, which we call $\nabla^n_T \psi$.
The operator $\nabla^n _T$ is tensorial in $T$ and differential of order $n$ in $\psi$, with symbol $\mathcal{S} (T)$, the totally symmetric part of $T$.
In coordinates, it is written as
\begin{equation*}
	\nabla^n_T \psi = T^{\mu_1 \dots \mu_n} \nabla_{\mu_1} \dots \nabla_{\mu_n} \psi \, .
\end{equation*}


One can think of this construction as a right inverse of $\sigma_n$: if $T$ is totally symmetric, then the symbol of $\nabla^n_T$ is $T$ itself.
Instead, given $D$ of order $n$, the operator $\nabla^n_{\sigma_{n} (D)}$ may very well be different from $D$ itself.
However, $D - \nabla^n_{\sigma_{n} (D)}$ is a differential operator of order at most $n$, but since its $n$-th symbol vanishes its order is actually strictly lower.
This motivates the following definition.

\begin{Definition}
	If $D$ is a differential operator of finite order $m > n$ on $E$, its $n$-th symbol is defined recursively as
	\begin{equation*}
		\sigma_{n}(D) := \sigma_n \Bigl( D - \sum_{j = n+1}^{m} \nabla^{j}_{\sigma_{j} (D)} \Bigr) \, .
	\end{equation*}
	We call the total symbol of $D$ the formal sum of all its symbols.
\end{Definition}

It follows from the definition that if $D$ has order $m$ then
\begin{equation*}
	D = \sum_{n=0}^{m} \nabla^n_{\sigma_n(D)} = \sum_{n = 0}^{\infty} \nabla^{n}_{\sigma_{n} (D)} \, .
\end{equation*}
Consequently, such an operator is completely determined by its symbols.

\subsection*{Section- and operator-valued forms on \texorpdfstring{$\Teich$}{Teich}}
\label{sec:operator_forms}

In the previous sections, frequent reference is made to forms on $\Teich$ valued in spaces of smooth sections of vector bundles over $\ModSp$ or differential operators acting on them.
While these spaces are infinite-dimensional, and not even necessarily coming with a topology, smoothness of such objects can be made sense of as follows.

\begin{Definition}
	\label{def:operator_forms}
	Suppose that $E \to \ModSp$ is a vector bundle, $\operatorname{Diff}(E)$ the space of finite-order differential operators acting on its sections.
	We say that a map $\psi \colon \Teich \to \smth(\ModSp,E)$ is smooth if it is as a section of the pull-back bundle of $E$ on $\ModSp \times \Teich$.
 	We also say that $D \colon \Teich \to \operatorname{Diff}(E)$ is smooth if, for every $\psi \in \smth(\ModSp,E)$, the map $D \psi \colon \Teich \to \smth(\ModSp,E)$ is smooth.
 	Similarly, we define $p$-differential forms valued in these spaces by pulling back the bundle of $p$-forms on $\Teich$ to $\ModSp \times \Teich$ and considering smooth sections thereof.
\end{Definition}

If $\psi \colon \Teich \to \smth(\ModSp,E)$ is smooth, $V$ a vector tangent to $\Teich$, the derivative $V[\psi]$ makes sense point-wise on $\ModSp$ and defines a smooth section $V[\psi]$ of $E$.
If $D$ is an operator-valued map, its derivative can be defined as
\begin{equation*}
	V[D] \psi := V \bigl[ D \psi \bigr] - D (V[\psi]) \, .
\end{equation*}
This derivative can easily be extended to an exterior differential.

A bracket $[ \cdot \wedge \cdot ]$ can be defined for operator-valued forms by regarding $\operatorname{Diff}(E)$ as a Lie algebra, for which the following rules apply.

\begin{Lemma}
	\label{lemma:twisteddiff}
	Let $\varphi$, $\psi$ and $\rho$ be $\End(E)$-valued differential forms of rank $a$, $b$ and $c$ respectively.
	Then the following hold:
	\begin{equation*}
		\begin{gathered}
			d^{\nabla} [ \varphi \wedge \psi ] = \big[ (d^{\nabla} \varphi) \wedge \psi \big] + (-1)^{a} \big[ \varphi \wedge d^{\nabla} \psi \big] \, ,	\\
			[ \varphi \wedge \psi ] = - (-1)^{ab} [ \psi \wedge \varphi] \, , \\
			(-1)^{ac} \big[ \varphi \wedge [ \psi \wedge \rho ] \big] + (-1)^{ba} \big[ \psi \wedge [ \rho \wedge \varphi ] \big] + (-1)^{cb} \big[ \rho \wedge [ \varphi \wedge \psi ] \big] = 0 \, .
		\end{gathered}
	\end{equation*}
	We shall refer to the last relation as the Jacobi identity.
\end{Lemma}
		

\phantomsection

\end{document}